\documentclass{amsart}

\newtheorem{thm}{Theorem}
\newtheorem{defn}{Definition}
\newtheorem{cor}{Corollary}
\newtheorem{conj}{Conjecture}
\newtheorem{rem}{Remark}

\begin{document}

\title[Root Lattice Point Counts in Hypercubes]{Point Counts of $D_k$ and Some $A_k$ and $E_k$ Integer Lattices Inside Hypercubes}

\author{Richard J. Mathar}
\urladdr{http://www.strw.leidenuniv.nl/~mathar}
\email{mathar@strw.leidenuniv.nl}
\address{Leiden Observatory, Leiden University, P.O. Box 9513, 2300 RA Leiden, The Netherlands}


\date{\today}

\begin{abstract}
Regular integer lattices are characterized by $k$ unit vectors
that build up their generator matrices. These have rank $k$ for $D$-lattices, and are
rank-deficient for $A$-lattices, for $E_6$ and $E_7$.
We count lattice points
inside hypercubes
centered at the origin for all three types, as if classified
by maximum infinity norm in the host lattice.
The results assume polynomial format as a function of the hypercube edge length.
\end{abstract}

\maketitle
\section{Scope} \label{sec.scop} 
We consider infinite translationally invariant point lattices set up
by generator matrices $G$
\begin{equation}
p_i=\sum_{j=1}^k G_{ij}\alpha_j
\end{equation}
which select point coordinates $p$ given a vector of integers $\alpha$.
In a purely geometric-enumerative manner we count all
points that reside inside a hypercube defined by $|p_i|\le n$, $\forall i$.
These numbers shall be called $A_k^b(n)$, $D_k^b(n)$ and $E_k^b(n)$
for the three lattice types dealt with.
In the incremental version of boxing the hypercubes, the points that
are on the surface of the hypercube are given the upper index $s$,
\begin{equation}
A_k^s(n)=A_k^b(n)-A_k^b(n-1), \,
D_k^s(n)=D_k^b(n)-D_k^b(n-1), \,
\mathrm{and}\,
E_k^s(n)=E_k^b(n)-E_k^b(n-1), 
\label{eq.ADEs}
\end{equation}
the first differences of the ``bulk'' numbers with respect to the edge size $n$.

There is vague resemblance to volume computation of the polytope defined
in $\alpha$-space by other straight cuts in $p$-space \cite{LawrenceMC57,Kasprzykarxiv10}.

In all cases discussed, the generating functions $D_k^b(x)$, $A_k^b(x)$
or $E_k^b(x)$ are rational functions with a factor $(1-x)^k$ in the
denominator.
They count sequences starting with a value of 1 at $n=0$\@.
The generating functions of the first differences, $D_k^s(x)$ etc.,
are therefore obtained by decrementing the exponent of $1-x$ in these denominators
by one \cite{Riordan,Wilf}, and have not been written down individually for that reason.

The manuscript considers first the $D$-lattices $D_6$--$D_4$ in tutorial detail in
sections \ref{sec.D2}--\ref{D4.sec}, then the case of general $k$ in
Section \ref{D5.sec}.
The points in $A_2$--$A_4$ are counted in sections \ref{A2.sec}--\ref{sec.A4}
by examining sums over the $\alpha$-coefficients,
and the general value of $k$ is addressed by summation over $p$-coordinates
in Section \ref{sec.A6}.
The cases $E_6$--$E_8$ are reduced to the earlier lattice counts
in sections \ref{sec.E6}--\ref{sec.E8}.

\section{Lattice $D_2$} \label{sec.D2} 

In the $D_2$ lattice, the expansion coefficients $\alpha_i$ and Cartesian
coordinates $p_i$ are connected by
\begin{equation}
\begin{pmatrix}
1 & 1 \\
1 & -1 \\
\end{pmatrix}
\cdot
\begin{pmatrix}
\alpha_1 \\
\alpha_2 \\
\end{pmatrix}
=
\begin{pmatrix}
p_1 \\
p_2 \\
\end{pmatrix}
.
\end{equation}
If we read the two lines of this system of equations separately, 
points inside the square $|p_i|\le n$ ($i=1,2$) are constrained to 
$\alpha$-coordinates inside a tilted square, as shown in Figure \ref{sq1.fig}.

\setlength{\unitlength}{0.4cm}
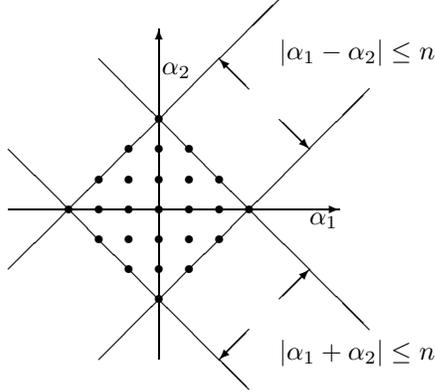
\begin{figure}[htb]
\begin{picture}(11,11)
\put(0,5){\vector(1,0){11}}
\put(10,4.5){$\alpha_1$}
\put(5,0){\vector(0,1){11}}
\put(5.1,9.5){$\alpha_2$}
\put(0,3){\line(1,1){9}}
\put(0,7){\line(1,-1){8}}
\put(3,10){\line(1,-1){9}}
\put(3,0){\line(1,1){9}}
\put(9,10){$|\alpha_1-\alpha_2|\le n$}
\put(8,9){\vector(-1,1){1}}
\put(9,8){\vector(1,-1){1}}
\put(9,0){$|\alpha_1+\alpha_2|\le n$}
\put(8,1){\vector(-1,-1){1}}
\put(9,2){\vector(1,1){1}}
\put(2,5){\circle*{0.3}}
\put(3,4){\circle*{0.3}}
\put(3,5){\circle*{0.3}}
\put(3,6){\circle*{0.3}}
\put(4,3){\circle*{0.3}}
\put(4,4){\circle*{0.3}}
\put(4,5){\circle*{0.3}}
\put(4,6){\circle*{0.3}}
\put(4,7){\circle*{0.3}}
\put(5,2){\circle*{0.3}}
\put(5,3){\circle*{0.3}}
\put(5,4){\circle*{0.3}}
\put(5,5){\circle*{0.3}}
\put(5,6){\circle*{0.3}}
\put(5,7){\circle*{0.3}}
\put(5,8){\circle*{0.3}}
\put(6,3){\circle*{0.3}}
\put(6,4){\circle*{0.3}}
\put(6,5){\circle*{0.3}}
\put(6,6){\circle*{0.3}}
\put(6,7){\circle*{0.3}}
\put(7,4){\circle*{0.3}}
\put(7,5){\circle*{0.3}}
\put(7,6){\circle*{0.3}}
\put(8,5){\circle*{0.3}}
\end{picture}
\caption{The conditions $|\alpha_1\pm \alpha_2|\le n$ select
two orthogonal diagonal stripes in the $(\alpha_1,\alpha_2)$-plane. Their
intersection is a tilted square centered at the origin.
}
\label{sq1.fig}
\end{figure}
The point count inside the square is
\begin{equation}
D_2^b = \sum_{|\alpha_1-\alpha_2|\le n}
\sum_{|\alpha_1+\alpha_2|\le n} 1.
\end{equation}
Resummation considering the two non-overlapping triangles below and above
the horizontal axis yields
\begin{multline}
D_2^b =
\sum_{\alpha_2=-n}^0
\sum_{\alpha_1=-n-\alpha_2}^{\alpha_2+n} 1
+
\sum_{\alpha_2=1}^n
\sum_{\alpha_1=\alpha_2-n}^{n-\alpha_2} 1
\label{eq.D2b}
\\
=
\sum_{\alpha_2=-n}^0
(2\alpha_2+2n+1)
+
\sum_{\alpha_2=1}^n
(2n-2\alpha_2+1).
\end{multline}
We will frequently sum over low order multinomials of this type with
a basic formula in terms of Bernoulli Polynomials $B$,
\cite[(0.121)]{GR}\cite[(1.2.11)]{ZwillingerCRC}\cite{EdwardsAMM93}
\begin{equation}
\sum_{m=1}^j m^k = \frac{B_{1+k}(j+1)-B_{1+k}(0)}{1+k}
.
\label{eq.Bern}
\end{equation}
Application to (\ref{eq.D2b}) and its first differences yields
essentially sequences A001844 and A008586 of the Online Encylopedia
of Integer Sequences (OEIS) \cite{EIS}:
\begin{thm}(Lattice points in the bulk and on the surface of $D_2$)
\begin{equation}
D_2^b(n)= 2n^2+2n+1 = 1,5,13,25,\ldots ;\quad D_2^s(n) = \begin{cases}
1,& n=0; \\
4n,& n>0.\\
\end{cases}
\end{equation}
\end{thm}

\section{Lattice $D_3$} \label{D3.sec} 

The relation between expansion coefficients $\alpha_i$ and Cartesian
coordinates $p_i$ for the $D_3$ lattice is
\begin{equation}
\begin{pmatrix}
1 & 1 & 0 \\
1 & -1 & 1 \\
0 & 0 & -1 \\
\end{pmatrix}
\cdot
\begin{pmatrix}
\alpha_1 \\
\alpha_2 \\
\alpha_3 \\
\end{pmatrix}
=
\begin{pmatrix}
p_1 \\
p_2 \\
p_3 \\
\end{pmatrix}
.
\end{equation}
The determinant of the Generator Matrix is non-zero; by multiplication
with the inverse matrix, a form more suitable to the counting problem 
results:
\begin{equation}
\begin{pmatrix}
1/2 & 1/2 & 1/2 \\
1/2 & -1/2 & -1/2 \\
0 & 0 & -1 \\
\end{pmatrix}
\cdot
\begin{pmatrix}
p_1 \\
p_2 \\
p_3 \\
\end{pmatrix}
=
\begin{pmatrix}
\alpha_1 \\
\alpha_2 \\
\alpha_3 \\
\end{pmatrix}
.
\label{D3inv.eq}
\end{equation}
$D_3^b(n)$ is the number of integer solutions restricted to the
cube $-n\le p_i\le n$. This is the full triple sum
$(2n+1)^3$---where $2n+1$ sizes the edge length of the cube---minus the number of
solutions of (\ref{D3inv.eq}) that result in non-integer $\alpha_i$.
The structure of the three equations in (\ref{D3inv.eq}) suggests
to separate the cases according to the parities of $p_3$ and $p_1+p_2$:
\begin{equation}
D_3^b(n)
= \sum_{
{ |p_1|\le n, |p_2|\le n, |p_3|\le n \atop p_1+p_2+p_3\,\mathrm{even}}
}
1
=
\sum_{ |p_1|\le n, |p_2|\le n\atop p_1+p_2\,\mathrm{even} }
\sum_{|p_3|\le n \atop p_3\,\mathrm{even} }
1
+
\sum_{ |p_1|\le n, |p_2|\le n\atop p_1+p_2\,\mathrm{odd} }
\sum_{ |p_3|\le n \atop p_3\,\mathrm{odd} }
1
.
\label{D3b1.eq}
\end{equation}
The auxiliary sums are examined separately for even and odd $n$
\cite[A109613,A052928]{EIS}:
\begin{eqnarray}
\sum_{ |p_3|\le n \atop p_3\,\mathrm{even} } 1
&=& n+\frac{1+(-1)^n}{2} = 1,1,3,3,5,5,7,7,9,9,\ldots; \label{V1g.eq}\\
\sum_{ |p_3|\le n \atop p_3\,\mathrm{odd} } 1
&=& n+\frac{1-(-1)^n}{2} = 0,2,2,4,4,6,6,8,8\ldots.
\label{V1u.eq}
\end{eqnarray}
The parity-filtered double sum of (\ref{D3b1.eq})
over the square in $(p_1,p_2)$-space
selects points on lines parallel to the diagonal.
\begin{defn}\label{V.def}(Order of even (g) and odd (u) point sets in $k$-dimensional hypercube planes)
\begin{equation}
V_k^g(n) \equiv
\sum_{ |p_i|\le n\atop p_1+p_2+\cdots p_k\,\mathrm{even} }1;
\quad
V^u_k(n) \equiv
\sum_{ |p_i|\le n\atop p_1+p_2+\cdots p_k\,\mathrm{odd} }1.
\label{V.eq}
\end{equation}
\end{defn}
This decomposition
applies
to higher dimensions recursively:
\begin{eqnarray}
V^g_k(n) &=& V^u_{k-1}(n)V^u_1(n)+V^g_{k-1}(n)V^g_1(n); \label{Vgrec.eq}\\
V^u_k(n) &=& V^u_{k-1}(n)V^g_1(n)+V^g_{k-1}(n)V^u_1(n).
\end{eqnarray}
Starting from $V^g_1(n)$ and $V^u_1(n)$ given in (\ref{V1g.eq})--(\ref{V1u.eq}),
the recurrences
provide Table \ref{V.tab}. The two disjoint sets of lattice points
complement
the hypercube:
\begin{equation}
V^g_k(n) + V^u_k(n) = (2n+1)^k.
\label{Vcompl.eq}
\end{equation}
\begin{table}
\caption{Low-dimensional examples of the lattice sums
(\ref{V.eq}).}
\begin{tabular}{ll}
index & value \\
\hline
$V^g_1(n)$ & $n+\frac{1+(-)^n}{2}$
\\
$V^u_1(n)$ & $n+\frac{1-(-)^n}{2}$
\\
$V^g_2(n)$ & $2n^2+2n+1$ 
\\
$V^u_2(n)$ & $2n(n+1)$ 
\\
$V^g_3(n)$ & $4n^3+6n^2+3n+\frac{1+(-)^n}{2}$ 
\\
$V^u_3(n)$ & $4n^3+6n^2+3n+\frac{1-(-)^n}{2}$ 
\\
$V^g_4(n)$ & $8n^4+16n^3+12n^2+4n+1$
\\
$V^u_4(n)$ & $4n(n+1)(2n^2+2n+1)$
\\
$V^g_5(n)$ & $16n^5+40n^4+40n^3+20n^2+5n+\frac{1+(-)^n}{2}$
\\
$V^u_5(n)$ & $16n^5+40n^4+40n^3+20n^2+5n+\frac{1-(-)^n}{2}$
\\
$V^g_6(n)$ & $(2n^2+2n+1)(16n^4+32n^3+20n^2+4n+1)$
\\
$V^u_6(n)$ & $2n(n+1)(4n^2+2n+1)(4n^2+6n+3)$
\\
$V^g_7(n)$ & $64n^7+224n^6+336n^5+280n^4+140n^3+42n^2+7n+\frac{1+(-)^n}{2}$
\\
$V^u_7(n)$ & $64n^7+224n^6+336n^5+280n^4+140n^3+42n^2+7n+\frac{1-(-)^n}{2}$
\\
$V^g_8(n)$ & $128n^8+512n^7+896n^6+896n^5+560n^4+224n^3+56n^2+8n+1$
\\
$V^u_8(n)$ & $8n(n+1)(2n^2+2n+1)(8n^4+16n^3+12n^2+4n+1)$
\end{tabular}
\label{V.tab}
\end{table}
\begin{thm}\label{V.thm}(fcc lattice counts for edge measure $2n+1$)
\begin{equation}
V_k^g(n) = \begin{cases}
\dfrac{(2n+1)^k}{2}+\dfrac{1}{2}, & k\,\mathrm{even}; \\
\dfrac{(2n+1)^k}{2}+\dfrac{(-)^n}{2}, & k\,\mathrm{odd}. \\
\end{cases}
\end{equation}
\end{thm}
\begin{proof}
The proof is simple by induction with the aid of
(\ref{Vgrec.eq}) and (\ref{Vcompl.eq}),
using $V_1^g(n)$ of (\ref{V1g.eq}) and $V_1^u(n)$ of (\ref{V1u.eq}).
\end{proof}

$D_3^b(n)$ in (\ref{D3b1.eq}) equals $V^g_3(n)$ by definition. $D_3^s$ and $D_3^b$ are sequences
A110907 and A175109 in the OEIS \cite{EIS}.
\begin{thm}(Lattice points in the bulk and on the surface of $D_3$)
\begin{eqnarray}
D_3^b(n) & =& 4n^3+6n^2+3n+\frac{1+(-)^n}{2} = 1,13,63,171,365,665\ldots ; \\
\quad D_3^s(n)  & =& \begin{cases}
1,& n=0; \\
12n^2+1+(-1)^n& n>0.\\
\end{cases}
= 1,12,50,108,194,300,434,\ldots
\end{eqnarray}
\end{thm}
The corresponding recurrences and generating function are
\begin{gather}
D_3^b(n)  = 3D_3^b(n-1)
-2D_3^b(n-2)
-2D_3^b(n-3)
+3D_3^b(n-4)
-D_3^b(n-5) ; \\
D_3^b(x) =\frac{(1+6x+x^2)(1+4x+x^2)}{(1+x)(1-x)^4}; \\
D_3^s(n)  = 2D_3^s(n-1)
-2D_3^s(n-3)
+D_3^s(n-4);\quad (n>3).
\end{gather}

\section{Lattice $D_4$} \label{D4.sec} 

The transformation between expansion coefficients and Cartesian
coordinates in the $D_4$ case reads
\begin{equation}
\begin{pmatrix}
1 & 1 & 0 & 0 \\
1 & -1 & 1 & 0 \\
0 & 0 & -1 & 1 \\
0 & 0 & 0 & -1 \\
\end{pmatrix}
\cdot
\begin{pmatrix}
\alpha_1 \\
\alpha_2 \\
\alpha_3 \\
\alpha_4 \\
\end{pmatrix}
=
\begin{pmatrix}
p_1 \\
p_2 \\
p_3 \\
p_4 \\
\end{pmatrix}
.
\end{equation}
The technique of counting points inside cubes is the same as
in the previous section. Inversion of the $4\times 4$ matrix yields
\begin{equation}
\begin{pmatrix}
1/2 & 1/2 & 1/2 & 1/2 \\
1/2 & -1/2 & -1/2 & -1/2 \\
0 & 0 & -1 & 1 \\
0 & 0 & 0 & -1 \\
\end{pmatrix}
\cdot
\begin{pmatrix}
p_1 & \\
p_2 & \\
p_3 & \\
p_4 & \\
\end{pmatrix}
=
\begin{pmatrix}
\alpha_1 \\
\alpha_2 \\
\alpha_3 \\
\alpha_4 \\
\end{pmatrix}
.
\end{equation}
We wish to count all lattice points subject
to the constraint $|p_i|\le n$ ($i=1,\ldots 4$), and the first two
lines of the previous equation require in addition that the
sum over all four $p_i$ is even to keep all four $\alpha_i$ in the
integer domain:
\begin{equation}
D_4^b(n)
= \sum_{
|p_1|\le n, |p_2|\le n, |p_3|\le n, |p_4|\le n \atop
p_1+p_2+p_3+p_4\,\mathrm{even}
}
1
.
\end{equation}
This expression is $V_4^g(n)$ already computed above. $D_4^s(n)$
is OEIS sequence A117216; $D_4^b(n)$ is A175110 \cite{EIS}.
\begin{thm}(Lattice points in the bulk and on the surface of $D_4$)
\begin{eqnarray}
D_4^b(n) & =& 1+4n+12n^2+16n^3+8n^4 \\
  &=& 1,41,313,1201,3281,7321,14281,25313,41761,65161,97241\ldots ; \nonumber \\
\quad D_4^s(n)  & =& \begin{cases}
1,& n=0; \\
8n(1+4n^2) & n>0;\\
\end{cases}
\\
&=& 1,40,272,888,2080,4040,6960,11032,16448,23400,32080\ldots \nonumber
\end{eqnarray}
\end{thm}
The associated generating function and recurrences are
\begin{gather}
D_4^b(x)  = \frac{1+36x+118x^2+36x^3+x^4}{(1-x)^5} ;\\
D_4^b(n)  = 5D_4^b(n-1) -10D_4^b(n-2) +10D_4^b(n-3)
-5D_4^b(n-4) +D_4^b(n-5);\\
D_4^s(n)   = 4D_4^s(n-1) -6D_4^s(n-2) +4D_4^s(n-3) -D_4^s(n-4);\quad (n>4).
\end{gather}

\section{Lattices $D_k$, general $k$ } \label{D5.sec} 
No new aspect arises in comparison to the previous two sections
\cite{SerraSagrIPL76}.
The $D_k^b(n)$ equal the $V_k^g(n)$ and their first differences 
constitute the $D_k^s(n)$:
\begin{gather}
D_5^b(n)= 16n^5+40n^4+40n^3+20n^2+5n+\frac{1+(-)^n}{2} ; \\
D_5^s(n)= \begin{cases}
1, & n=0;\\
1+40n^2+80n^4+(-)^n, & n>0;
\end{cases}
 \\
D_6^b(n)= 32n^6+96n^5+120n^4 +80n^3 +30n^2 +6n+1; \\
D_6^s(n)= \begin{cases}
1, & n=0;\\
4n(1+12n^2)(3+4n^2), & n>0;
\end{cases}
 ; \\
D_7^b(n)= 64n^7+225n^6+336n^5+280n^4+130n^3+43n^2+7n+\frac{1+(-)^n}{2}; \\
D_7^s(n)= \begin{cases}
1, & n=0;\\
1+84n^2+560n^4+448n^6+(-)^n, & n>0.
\end{cases}
\end{gather}
$D_5$ and $D_6$ are materialized as sequences A175111 to A175114 \cite{EIS}.
All cases
are summarized in a Corollary to Theorem \ref{V.thm}:
\begin{cor}($D_k$ Lattice points inside the hypercube)
\begin{equation}
D_k^b(n) = \begin{cases}
\frac{(2n+1)^k}{2}+\frac{1}{2}, & k\,\mathrm{even}; \\
\frac{(2n+1)^k}{2}+\frac{(-)^n}{2}, & k\,\mathrm{odd}. \\
\end{cases}
\label{eq.Dkbn}
\end{equation}
\begin{equation}
D_k^s(n) = \begin{cases}
\frac{(2n+1)^k}{2}-\frac{(2n-1)^k}{2}, & k\,\mathrm{even},\,n>0; \\
\frac{(2n+1)^k}{2}-\frac{(2n-1)^k}{2}+(-)^n, & k\,\mathrm{odd},\,n>0. \\
\end{cases}
\end{equation}
\end{cor}
The generating functions are
\begin{equation}
D_k^b(x) = 
\begin{cases}
\dfrac{\sum_{i=0}^k \beta_i^g x^i}{(1-x)^{k+1}},& k\,\mathrm{even};\\
\dfrac{1+\sum_{i=1}^k \beta_i^u x^i}{(1+x)(1-x)^{k+1}},& k\,\mathrm{odd};\\
\end{cases}
\end{equation}
where
\begin{equation}
2\beta_i^g\equiv
\sum_{t=0}^{i}
[(2i-2t+1)^k+1]
\binom{k+1}{t}(-)^t
,
\end{equation}
\begin{equation}
2\beta_i^u\equiv
\sum_{t=0}^{i}
[(2i-2t+1)^k+(-)^{i-t}]
\binom{k+1}{t}(-)^t
+
\sum_{t=0}^{i-1}
[(2i-2t-1)^k-(-)^{i-t}]
\binom{k+1}{t}(-)^t
.
\end{equation}

\begin{rem}
The $D_k^*$ lattices are characterized by
\begin{equation}
\begin{pmatrix}
1 & 0 & 0& \cdots & 0 & 1/2 \\
0 & 1 & 0& \cdots & 0 & 1/2 \\
0 & 0 & 1& \ddots & 0 & 1/2 \\
\vdots & \vdots & 0& 1& \ddots & 1/2 \\
\vdots & \vdots & \vdots & 0& 1& 1/2 \\
0 & 0 & 0& \cdots & 0& 1/2 \\
\end{pmatrix}
\cdot
\begin{pmatrix}
\alpha_1 \\
\alpha_2 \\
\alpha_3 \\
\vdots \\
\end{pmatrix}
=
\begin{pmatrix}
p_1 \\
p_2 \\
p_3 \\
\vdots \\
\end{pmatrix}
.
\end{equation}
Matrix inversion gives
\begin{equation}
\begin{pmatrix}
1 & 0 & \cdots & 0 & -1 \\
0 & 1 & 0& \vdots & -1 \\
\vdots & 0 & 1& \vdots & -1 \\
0 & \vdots & \ddots & 1& -1 \\
0 & 0 & \cdots & 0& 2 \\
\end{pmatrix}
\cdot
\begin{pmatrix}
p_1 \\
p_2 \\
p_3 \\
\vdots \\
\end{pmatrix}
=
\begin{pmatrix}
\alpha_1 \\
\alpha_2 \\
\alpha_3 \\
\vdots \\
\end{pmatrix}
.
\end{equation}
which shows that there is no constraint on generating any $p_i$ inside
the regions $|p_i|\le n$: The number of lattice points up to infinity norm
$n$ is simply $D_k^{*b}(n)=(2n+1)^k$.
\end{rem}

\section{Lattice $A_2$} \label{A2.sec}
$A_2^b(n)$ is the number of integer solutions to
\begin{equation}
\begin{pmatrix}
1 & 0 \\
-1 & 1 \\
0 & -1 \\
\end{pmatrix}
\cdot
\begin{pmatrix}
\alpha_1 \\
\alpha_2 \\
\end{pmatrix}
=
\begin{pmatrix}
p_1 \\
p_2 \\
p_3 \\
\end{pmatrix}
\end{equation}
in the range $|p_i|\le n$.
The three requirements from the three lines of this equation become
\begin{equation}
A_2^b =
\sum_{|\alpha_1|\le n}\,\sum_{|-\alpha_1+\alpha_2|\le n\atop |-\alpha_2|\le n}1
.
\end{equation}
\begin{figure}[htb]
\begin{picture}(16,12)
\put(0,5){\vector(1,0){15}}
\put(0,8){\line(1,0){13}}
\put(0,2){\line(1,0){13}}
\put(14,4.5){$\alpha_1$}
\put(5,0){\vector(0,1){11}}
\put(2,0){\line(0,1){11}}
\put(8,0){\line(0,1){11}}
\put(5.1,9.5){$\alpha_2$}
\put(0,3){\line(1,1){9}}
\put(3,0){\line(1,1){9}}
\put(9,10){$|-\alpha_1+\alpha_2|\le n$}
\put(8.5,9.5){\vector(-1,1){1}}
\put(9.5,8.5){\vector(1,-1){1}}
\put(12,6){\vector(0,1){2}}
\put(12,4){\vector(0,-1){2}}
\put(12,6.5){$|\alpha_2|\le n$}
\put(2,2){\circle*{0.3}}
\put(2,3){\circle*{0.3}}
\put(2,4){\circle*{0.3}}
\put(2,5){\circle*{0.3}}
\put(3,2){\circle*{0.3}}
\put(3,3){\circle*{0.3}}
\put(3,4){\circle*{0.3}}
\put(3,5){\circle*{0.3}}
\put(3,6){\circle*{0.3}}
\put(4,2){\circle*{0.3}}
\put(4,3){\circle*{0.3}}
\put(4,4){\circle*{0.3}}
\put(4,5){\circle*{0.3}}
\put(4,6){\circle*{0.3}}
\put(4,7){\circle*{0.3}}
\put(5,2){\circle*{0.3}}
\put(5,3){\circle*{0.3}}
\put(5,4){\circle*{0.3}}
\put(5,5){\circle*{0.3}}
\put(5,6){\circle*{0.3}}
\put(5,7){\circle*{0.3}}
\put(5,8){\circle*{0.3}}
\put(6,3){\circle*{0.3}}
\put(6,4){\circle*{0.3}}
\put(6,5){\circle*{0.3}}
\put(6,6){\circle*{0.3}}
\put(6,7){\circle*{0.3}}
\put(6,8){\circle*{0.3}}
\put(7,4){\circle*{0.3}}
\put(7,5){\circle*{0.3}}
\put(7,6){\circle*{0.3}}
\put(7,7){\circle*{0.3}}
\put(7,8){\circle*{0.3}}
\put(8,5){\circle*{0.3}}
\put(8,6){\circle*{0.3}}
\put(8,7){\circle*{0.3}}
\put(8,8){\circle*{0.3}}
\end{picture}
\caption{The conditions $|\alpha_1|\le n$ and $|\alpha_2|\le n$ select
a square in the $(\alpha_1,\alpha_2)$-plane. The 
requirement $|-\alpha_1+\alpha_2|\le n$ admits only values inside a diagonal
stripe. The intersection is the dotted hexagon.
}
\label{loz.fig}
\end{figure}
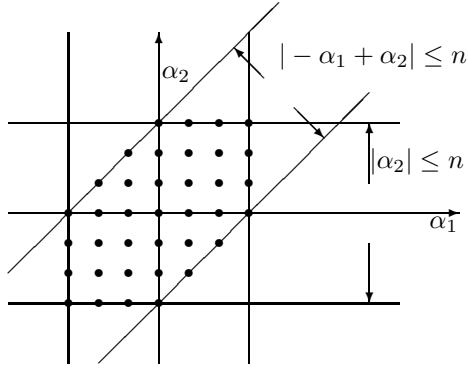
As outlined in Figure \ref{loz.fig}, decomposition of the conditions allows
resummation over the
quadrangles above and below the $\alpha_1$ axis:
\begin{equation}
A_2^b(n) =
\sum_{\alpha_2=-n}^0 \sum_{\alpha_1=-n}^{n+\alpha_2} 1
+
\sum_{\alpha_2=1}^n \sum_{\alpha_1=\alpha_2-n}^{n} 1
=
\sum_{\alpha_2=-n}^0 (2n+1+\alpha_2)
+
\sum_{\alpha_2=1}^n (2n+1-\alpha_2),
\end{equation}
further evaluated with \eqref{eq.Bern}.
\begin{thm}(Lattice points in the bulk and on the surface of $A_2$, \cite[A003215]{EIS})
\begin{equation}
A_2^b(n) = 1+3n(n+1) = 1,7,19,37,61,91,127,169,217,271,331,397,469,\ldots \\
\label{eq.A2}
\end{equation}
\end{thm}
The first differences are
\cite[A008458]{EIS}
\begin{equation}
A_2^s(n) = \begin{cases}
1, & n=0;\\
6n, & n>0.\\
\end{cases}
\end{equation}

\section{Lattice $A_3$}
The generator matrix sets
\begin{equation}
\begin{pmatrix}
1 & 0 & 0 \\
-1 & 1 & 0 \\
0 & -1 & 1 \\
0 & 0 & -1 \\
\end{pmatrix}
\cdot
\begin{pmatrix}
\alpha_1 \\
\alpha_2 \\
\alpha_3 \\
\end{pmatrix}
=
\begin{pmatrix}
p_1 \\
p_2 \\
p_3 \\
p_4 \\
\end{pmatrix}
.
\end{equation}
This translates the four bindings $|p_i|\le n$ into four
constraints on the three $\alpha$:
\begin{equation}
A_3^b(n) = 
\sum_{|\alpha_1|\le n}
\,
\sum_{|-\alpha_1+\alpha_2|\le n\atop |-\alpha_2+\alpha_3|\le n}
\,
\sum_{|-\alpha_3|\le n} 1.
\end{equation}
Figure \ref{kit.fig} illustrates resummation of the format
\begin{equation}
\sum_{|\alpha_1|\le n}
\sum_{|-\alpha_1+\alpha_2|\le n}
1
=
\sum_{\alpha_2=-2n}^0 \sum_{\alpha_1=-n}^{\alpha_2+n}1
+\sum_{\alpha_2=1}^{2n} \sum_{\alpha_1=\alpha_2-n}^{n}1
.
\label{resukit.eq}
\end{equation}
\begin{figure}[htb]
\begin{picture}(11,11)
\put(0,5){\vector(1,0){11}}
\put(10,4.5){$\alpha_1$}
\put(5,0){\vector(0,1){11}}
\put(2,0){\line(0,1){11}}
\put(8,0){\line(0,1){11}}
\put(5.1,9.5){$\alpha_2$}
\put(0,3){\line(1,1){9}}
\put(3,0){\line(1,1){9}}
\put(9,10){$|-\alpha_1+\alpha_2|\le n$}
\put(8,9){\vector(-1,1){1}}
\put(9,8){\vector(1,-1){1}}
\put(2,0){\circle*{0.3}}
\put(2,1){\circle*{0.3}}
\put(2,2){\circle*{0.3}}
\put(2,3){\circle*{0.3}}
\put(2,4){\circle*{0.3}}
\put(2,5){\circle*{0.3}}
\put(3,0){\circle*{0.3}}
\put(3,1){\circle*{0.3}}
\put(3,2){\circle*{0.3}}
\put(3,3){\circle*{0.3}}
\put(3,4){\circle*{0.3}}
\put(3,5){\circle*{0.3}}
\put(3,6){\circle*{0.3}}
\put(4,1){\circle*{0.3}}
\put(4,2){\circle*{0.3}}
\put(4,3){\circle*{0.3}}
\put(4,4){\circle*{0.3}}
\put(4,5){\circle*{0.3}}
\put(4,6){\circle*{0.3}}
\put(4,7){\circle*{0.3}}
\put(5,2){\circle*{0.3}}
\put(5,3){\circle*{0.3}}
\put(5,4){\circle*{0.3}}
\put(5,5){\circle*{0.3}}
\put(5,6){\circle*{0.3}}
\put(5,7){\circle*{0.3}}
\put(5,8){\circle*{0.3}}
\put(6,3){\circle*{0.3}}
\put(6,4){\circle*{0.3}}
\put(6,5){\circle*{0.3}}
\put(6,6){\circle*{0.3}}
\put(6,7){\circle*{0.3}}
\put(6,8){\circle*{0.3}}
\put(6,9){\circle*{0.3}}
\put(7,4){\circle*{0.3}}
\put(7,5){\circle*{0.3}}
\put(7,6){\circle*{0.3}}
\put(7,7){\circle*{0.3}}
\put(7,8){\circle*{0.3}}
\put(7,9){\circle*{0.3}}
\put(7,10){\circle*{0.3}}
\put(8,5){\circle*{0.3}}
\put(8,6){\circle*{0.3}}
\put(8,7){\circle*{0.3}}
\put(8,8){\circle*{0.3}}
\put(8,9){\circle*{0.3}}
\put(8,10){\circle*{0.3}}
\put(8,11){\circle*{0.3}}
\end{picture}
\caption{The conditions $|\alpha_1|\le n$ and
$|-\alpha_1+\alpha_2|\le n$ select points in the dotted parallelogram.
}
\label{kit.fig}
\end{figure}
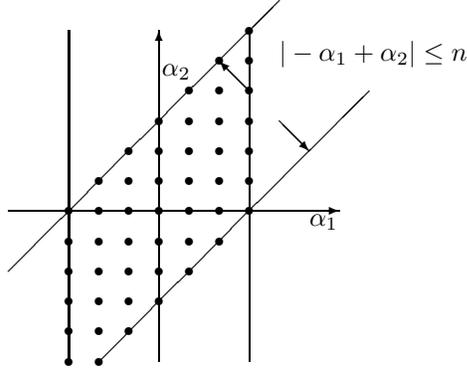
This is applied twice (note this factorization generates quad-sums which are
a convenient notation to keep track of the limits. The sums actually remain triple sums):
\begin{multline}
A_3^b(n) =
(\sum_{\alpha_2=-2n}^0\sum_{\alpha_1=-n}^{\alpha_2+n}1
+\sum_{\alpha_2=1}^{2n}\sum_{\alpha_1=\alpha_2-n}^{n}1
)
(\sum_{\alpha_2=-2n}^0\sum_{\alpha_3=-n}^{\alpha_2+n}1
+\sum_{\alpha_2=1}^{2n}\sum_{\alpha_3=\alpha_2-n}^{n}1
)
\\
=
\sum_{\alpha_2=-2n}^0\sum_{\alpha_1=-n}^{\alpha_2+n}
\sum_{\alpha_3=-n}^{\alpha_2+n}1
+\sum_{\alpha_2=1}^{2n}\sum_{\alpha_1=\alpha_2-n}^{n}
\sum_{\alpha_3=\alpha_2-n}^{n}1
\\
=
\sum_{\alpha_2=-2n}^0 (2n+1+\alpha_2)^2
+
\sum_{\alpha_2=1}^{2n} (2n+1-\alpha_2)^2.
\end{multline}
After binomial expansion, both remaining sums are reduced with (\ref{eq.Bern}):
\begin{thm}(Lattice points in the bulk and on the surface of $A_3$)
\begin{equation}
A_3^b(n)
= 1+\frac{2}{3}n(7+12n+8n^2)
= 1,19,85,231,489,891,1469,2255,3281,\ldots
\label{eq.A3}
\end{equation}
\begin{equation}
A_3^s(n) = \begin{cases}
1, & n=0\\
2+16n^2, & n>0\\
\end{cases}
=1,18,66,146,258,402,578,\ldots
\end{equation}
\end{thm}
These are sequences A063496 and A010006 in the OEIS \cite{EIS}.

\section{Lattice $A_4$}\label{sec.A4}
$A_4$ is characterized by a quad-sum over $\alpha_i$ with five
constraints on the $p_i$ set up by the hypercube:
\begin{equation}
\begin{pmatrix}
1 & 0 & 0 & 0 \\
-1 & 1 & 0 & 0 \\
0 & -1 & 1 & 0 \\
0 & 0 & -1 & 1 \\
0 & 0 & 0 & -1 \\
\end{pmatrix}
\cdot
\begin{pmatrix}
\alpha_1 \\
\alpha_2 \\
\alpha_3 \\
\alpha_4 \\
\end{pmatrix}
=
\begin{pmatrix}
p_1 \\
p_2 \\
p_3 \\
p_4 \\
p_5 \\
\end{pmatrix}
.
\end{equation}
\begin{equation}
A_4^b(n)=
\sum_{|\alpha_1|\le n}
\,
\sum_{|-\alpha_1+\alpha_2|\le n}
\,
\sum_{|-\alpha_2+\alpha_3|\le n\atop |-\alpha_3+\alpha_4|\le n}
\,
\sum_{|\alpha_4|\le n} 1.
\end{equation}
The resummation (\ref{resukit.eq}) is separately applied to $(\alpha_1,\alpha_2)$
and
$(\alpha_3,\alpha_4)$; the entanglement between $\alpha_2$ and $\alpha_3$
is noted in the second factor:
\begin{multline}
A_4^b(n)
=
\left(
\sum_{\alpha_2=-2n}^0
\sum_{\alpha_1=-n}^{\alpha_2+n}
1
+
\sum_{\alpha_2=1}^{2n}
\sum_{\alpha_1=\alpha_2-n}^n
1\right)
\\ \times
\left(
\sum_{\alpha_3=-2n \atop |-\alpha_2+\alpha_3|\le n}^0
\sum_{\alpha_4=-n}^{\alpha_3+n}
1
+
\sum_{\alpha_3=1\atop |-\alpha_2+\alpha_3|\le n}^{2n}
\sum_{\alpha_4=\alpha_3-n}^n
1\right)
\\
=
\left(
\sum_{\alpha_2=-2n}^0
(2n+1+\alpha_2)
+
\sum_{\alpha_2=1}^{2n}
(2n+1-\alpha_2)\right)
\\ \times
\left(
\sum_{\alpha_3=-2n \atop |-\alpha_2+\alpha_3|\le n}^0
(2n+1+\alpha_3)
+
\sum_{\alpha_3=1\atop |-\alpha_2+\alpha_3|\le n}^{2n}
(2n+1-\alpha_3)
\right)
.
\label{A4fin.eq}
\end{multline}
Product expansion generates 4 terms. The coupling between $\alpha_2$ and $\alpha_3$
is rewritten individually in their 4 different quadrants
facilitated with Figure \ref{hex2.fig}.
\begin{figure}[htb]
\begin{picture}(18,13)
\put(3,6){\vector(1,0){17}}
\put(3,10){\line(1,0){16}}
\put(3,2){\line(1,0){16}}
\put(19,5.5){$\alpha_2$}
\put(10,0){\vector(0,1){14}}
\put(6,0.5){\line(0,1){12}}
\put(14,0.5){\line(0,1){12}}
\put(10.1,12.5){$\alpha_3$}
\put(3,1){\line(1,1){12}}
\put(6,0){\line(1,1){11}}
\put(6,2){\circle*{0.3}}
\put(6,3){\circle*{0.3}}
\put(6,4){\circle*{0.3}}
\put(7,2){\circle*{0.3}}
\put(7,3){\circle*{0.3}}
\put(7,4){\circle*{0.3}}
\put(7,5){\circle*{0.3}}
\put(8,2){\circle*{0.3}}
\put(8,3){\circle*{0.3}}
\put(8,4){\circle*{0.3}}
\put(8,5){\circle*{0.3}}
\put(8,6){\circle*{0.3}}
\put(9,3){\circle*{0.3}}
\put(9,4){\circle*{0.3}}
\put(9,5){\circle*{0.3}}
\put(9,6){\circle*{0.3}}
\put(9,7){\circle*{0.3}}
\put(10,4){\circle*{0.3}}
\put(10,5){\circle*{0.3}}
\put(10,6){\circle*{0.3}}
\put(10,7){\circle*{0.3}}
\put(10,8){\circle*{0.3}}
\put(11,5){\circle*{0.3}}
\put(11,6){\circle*{0.3}}
\put(11,7){\circle*{0.3}}
\put(11,8){\circle*{0.3}}
\put(11,9){\circle*{0.3}}
\put(12,6){\circle*{0.3}}
\put(12,7){\circle*{0.3}}
\put(12,8){\circle*{0.3}}
\put(12,9){\circle*{0.3}}
\put(12,10){\circle*{0.3}}
\put(13,7){\circle*{0.3}}
\put(13,8){\circle*{0.3}}
\put(13,9){\circle*{0.3}}
\put(13,10){\circle*{0.3}}
\put(14,8){\circle*{0.3}}
\put(14,9){\circle*{0.3}}
\put(14,10){\circle*{0.3}}
\end{picture}
\caption{The conditions $|\alpha_2|\le 2n$ and
$|\alpha_3|\le 2n$ define the large square, and
$|-\alpha_2+\alpha_3|\le n$ narrows the region down to the dotted hexagon.
}
\label{hex2.fig}
\end{figure}
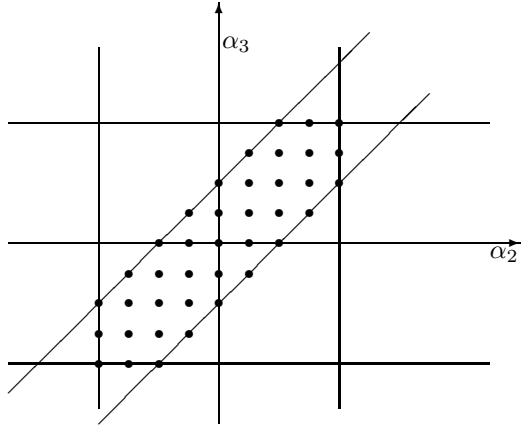
\begin{gather}
\sum_{\alpha_2=-2n}^0
\sum_{\alpha_3=-2n\atop |-\alpha_2+\alpha_3|\le n}^0
= 
\sum_{\alpha_3=-2n}^{-n}
\sum_{\alpha_2=-n}^{\alpha_3+n}
+
\sum_{\alpha_3=-n+1}^{0}
\sum_{\alpha_2=\alpha_3-n}^{0};\\
\sum_{\alpha_2=-2n}^0
\sum_{\alpha_3=1\atop |-\alpha_2+\alpha_3|\le n}^{2n}
= 
\sum_{\alpha_3=1}^n
\sum_{\alpha_2=\alpha_3-n}^0; \\
\sum_{\alpha_2=1}^{2n}
\sum_{\alpha_3=-2n\atop |-\alpha_2+\alpha_3|\le n}^0
= 
\sum_{\alpha_3=-n+1}^0
\sum_{\alpha_2=1}^{\alpha_3+n}; \\
\sum_{\alpha_2=1}^{2n}
\sum_{\alpha_3=1\atop |-\alpha_2+\alpha_3|\le n}^{2n}
= 
\sum_{\alpha_3=1}^n
\sum_{\alpha_2=1}^{\alpha_3+n}
+
\sum_{\alpha_3=n+1}^{2n}
\sum_{\alpha_2=\alpha_3-n}^{2n}.
\end{gather}
So $A_4^b$ in \eqref{A4fin.eq} translates into six elementary double sums over products of the
form $(2n+1\pm \alpha_2)(2n+1\pm \alpha_3)$, eventually handled with \eqref{eq.Bern}.
\begin{thm} (Lattice points in the bulk and on the surface of $A_4$ \cite[A083669]{EIS})
\begin{gather}
A_4^b(n) = 1+\frac{5}{12}n(n+1)(14+23n+23n^2)
 = 1, 51, 381, 1451, 3951, 8801, \ldots
\label{eq.A4}
\\
A_4^s(n) = \frac{5}{3}(7+23n^2)
 = 1,50,330,1070,2500,4850,8350\ldots
\end{gather}
\end{thm}

\section{Lattices $A_k$, $k>4$} \label{sec.A6}
Direct summation over the polytopes in $\alpha_i$-space becomes
increasingly laborious in higher dimensions; we switch to summation
in $p_i$-space based on the alternative
\begin{equation}
A_k^b(n) =
\sum_{|p_i|\le n,\, i=1,\ldots,k+1 \atop \sum_{i=1}^{k+1} p_i=0} 1.
\end{equation}
This is derived by adding a $k+1$-st unit vector with all components
equal to zero---with the exception of the last component--as a
final column to the generator matrix. (This
simple format suffices; laminations are not involved.)
An associated coefficient $\alpha_{k+1}$
embeds the lattice into full space, while
the condition $\alpha_{k+1}=0$ is maintained for the counting
process. Inversion of the matrix generator equation demonstrates that 
this zero-condition translates into the requirement on the sum
over the $p_i$ shown above. This point of view is occasionally used
to \emph{define} the $A$-lattices.

Counting the points subjected to some fixed $\sum_i p_i=m$ is equivalent to computation
of the multinomial coefficient
\begin{equation}
[x^m] (1+x+x^{-1}+x^2+x^{-2}+\cdots +x^n+x^{-n})^k.
\end{equation}
Balancing the accumulated powers 
as required for $A_k^b$ necessarily ties them to the central multinomial numbers
\cite{Belbachirarxiv07,BelbachirAMI35}:
\begin{equation}
A_{k-1}^b(n) = \binom{k}{nk}_{2n}
= \sum_{j=0}^{\lfloor nk/(2n+1) \rfloor}
(-)^j\binom{k}{j}
\binom{nk-j(2n+1)+k-1}{k-1}
.
\label{eq.Belb}
\end{equation}
Selecting values for the $p_i$ is equivalent to a Motzkin-path, picking
one term of each of the $k$ instances of the $1+x+x^{-1}$ of the trinomial,
for example
\cite{BlasiakJIS11}.
First, the formula is a route to quick numerical evaluation (Table \ref{Abconv.tab}).
Second, it proves that $A_k^b(n)$ is a polynomial of order $\le k$
in $n$, because each of the binomial factors in the $j$-sum is a polynomial
of order $k-1$. This is easily made more explicit by invocation of
the Stirling numbers of the first kind \cite{RiordanAMM71}\cite[(24.1.3)]{AS}.
\begin{rem}
This scheme of polynomial extension has
been used for coordination sequences before \cite{ConwayPRSC453},
and is found in growth series as well \cite{Ardilaarxiv08}.
\end{rem}

\begin{table}[htb]
\caption{$A_k^b(n)$
displaying columns of central 3-nomial, 5-nomial, 7-nomial  etc.\ numbers
\cite[A002426,A005191,A025012,A025014,A163269]{EIS}
}
\begin{tabular}{r|rrrrrrrrr}
$k\backslash n$ & 0 & 1 & 2 &3 &4 &5 &6 &7 & 8\\
\hline
1 & 1&3&5&7&9&11&13&15&17 \\
2 & 1&7&19&37&61&91&127&169&217 \\
3 & 1&19&85&231&489&891&1469&2255&3281 \\
4 & 1&51&381&1451&3951&8801&17151&30381&50101 \\
5 & 1&141&1751&9331&32661&88913&204763&418503&782153 \\
6 & 1&393&8135&60691&273127&908755&2473325&5832765&12354469 \\
7 & 1&1107&38165&398567&2306025&9377467&30162301&82073295&197018321 \\
8 & 1&3139&180325&2636263&19610233&97464799&370487485&1163205475&3164588407
\end{tabular}
\label{Abconv.tab}
\end{table}

By computing the initial terms of any $A_k$
numerically, the others follow by the
recurrence obeyed by $k$-th order polynomials \cite{FilipponiUQ2}:
\begin{equation}
A_k^b(n) = \sum_{j=1}^{k+1} \binom{k+1}{j} (-)^{j+1} A_k^b(n-j).
\end{equation}
\begin{thm}(Lattice points in the bulk and on the surface of $A_5$)
\begin{eqnarray}
A_5^b(n)
&=& \frac{1}{5}(2n+1)(5+27n+71n^2+88n^3+44n^4)
;
\label{eq.A5}
\\
A_5^s(n) &=& \begin{cases}
1, & n=0,\\
2+50n^2+88n^4, & n>0,
\end{cases}
=
1,140,1610,7580, 23330,\ldots
\end{eqnarray}
\end{thm}
$A_5^b$ is a bisection of sequence A071816 of the OEIS \cite{EIS}.
$A_6^b$ is a bisection of sequence A133458 \cite{EIS}.
\begin{thm}($A_6$ and $A_7$ point counts)
\begin{gather}
A_6^b(n)
= 1+\frac{7}{180}n(n+1)(222+727n+1568n^2+1682n^3+841n^4)
.
\\
A_6^s(n) = \begin{cases}
1, & n=0,\\
\frac{7}{30}n(74+765n^2+841n^4), & n>0,\\
\end{cases} = 1,392,7742,52556,212436\ldots
\\
A_7^b(n)
=\frac{2n+1}{315}(315+2568n+10936n^2+26400n^3+37360n^4+28992n^5+9664n^6)
.
\end{gather}
\end{thm}

\begin{rem}
The $A_k^b(n)$ can be phrased as $k$-th order polynomials of
$L\equiv 2n+1$ with the same parity as $k$:
\begin{eqnarray}
A_1^b(L) &=& L; \\
A_2^b(L) &=& \frac{1}{4}+\frac{3}{4}L^2; \\
A_3^b(L) &=& \frac{1}{3}L+\frac{2}{3}L^3; \\
A_4^b(L) &=& \frac{9}{64}+\frac{25}{96}L^2+\frac{115}{192}L^4; \\
A_5^b(L) &=& \frac{1}{5}L+\frac{1}{4}L^3+\frac{11}{20}L^5; \\
A_6^b(L) &=& \frac{25}{256}+\frac{539}{2304}L^4+\frac{5887}{11520}L^6; \\
A_7^b(L) &=& \frac{1}{7}L+\frac{7}{45}L^3+\frac{2}{9}L^5+\frac{151}{315}L^7; \\
A_8^b(L) &=& \frac{1225}{16384}+\frac{3229}{28672}L^2+\frac{6063}{40960}L^4
  +\frac{867}{4096}L^6+\frac{259723}{573440}L^8.
\end{eqnarray}
\end{rem}
If we rewrite \eqref{eq.Belb} \cite{RoyAMM94}
\begin{equation}
A_{k-1}^b(n) = 
\sum_{j=0}^{\lfloor k/(2+1/n)\rfloor}
(-1)^j \frac{k}{j!} \frac{\Gamma[k(n+1)-j(2n+1)]}{\Gamma(k-j+1)\Gamma[kn-j(2n+1)+1]},
\label{eq.Akgamma}
\end{equation}
the multiplication formula
of the $\Gamma$-function converts
this
to terminating
Saalsch\"utzian
Hypergeometric Series:
\begin{gather}
A_{k-1}^b(1) = \frac{\Gamma(2k)}{\Gamma(k)\Gamma(k+1)}
\,_4F_3\left( \begin{array}{cc}-k,-\frac{k}{3},-\frac{k-1}{3},
-\frac{k-2}{3}\\
-\frac{2k-1}{3},-\frac{2k-2}{3},-\frac{2k}{3}+1
\end{array}\mid 1
\right)
,
\\
A_{k-1}^b(n) = \frac{\Gamma[(n+1)k]}{\Gamma(k)\Gamma(nk+1)}
\,_{2n+2}F_{2n+1}\left( \begin{array}{cc}-k,-\frac{nk}{2n+1},-\frac{nk-1}{2n+1},
-\frac{nk-2}{2n+1},\cdots,-\frac{nk-2n}{2n+1}\\
-\frac{(n+1)k-1}{2n+1},-\frac{(n+1)k-2}{2n+1},\cdots,-\frac{(n+1)k-2n-1}{2n+1}
\end{array}\mid 1
\right)
.
\end{gather}

The functional equation $\Gamma(m+1)=m\Gamma(m)$ presumably induces a non-linear
recurrence along each column of Table \ref{Abconv.tab}, as
shown by Sulanke for column $n=1$ \cite{SulankeJIS3}.
Numerical experimentation rather than proofs \cite{RieseJSC35} suggest:
\begin{conj} (Recurrences of centered 3-nomial, 5-nomial, 7-nomial coefficients)\label{conj.Akb}
\begin{gather}
(k+1)A_k^b(1) -(2k+1)A_{k-1}^b(1)-3kA_{k-2}^b(1) =0;\\
\begin{split}
2(k+1)(2k+1)A_k^b(2)+(k^2-49k-2)A_{k-1}^b(2)+5(-21k^2+37k-18)A_{k-2}^b(2)
\\
 -25(k-1)(k-4)A_{k-3}^b(2)+125(k-1)(k-2)A_{k-4}^b(2)=0.
\end{split}
\\
\begin{split}
3(3k+2)(3k+1)(k+1)A_k^b(3)
+(41k^3-600k^2-191k-6)A_{k-1}^b(3)
\\
+7(-383k^3+1458k^2-1927k+840)A_{k-2}^b(3)
 +49(-83k^3+1068k^2-4321k+5040)A_{k-3}^b(3)
\\
+343(199k^3-1890k^2+6017k-6390)A_{k-4}^b(3)
+2401(k-3)(43k^2-351k+722)A_{k-5}^b(3)
\\
-16807(k-3)(k-4)(5k-19)A_{k-6}^b(3)
-117649(k-5)(k-4)(k-3)A_{k-7}^b(3)
=0.
\end{split}
\end{gather}
\end{conj}
\begin{rem}
Inverse binomial transformations of the $A_k^b(n)$ define
coefficients $\eta_{k,j}$ via
\begin{gather}
A_k^b(n) \equiv 1+2\sum_{j=1}^n \binom{n}{j}\eta_{k,j},\\
\eta_{k,j} = \frac{1}{2}\sum_{l=0}^j (-)^{j+l}\binom{j}{l}\binom{k+1}{l(k+1)}_{2l},
\label{bin.eq}
\end{gather}
as demonstrated in Table \ref{bern.tab}.
They are related to the partial fractions of the rational generating functions :
\begin{equation}
A_k^b(x) = \frac{1}{1-x}+2\sum_{j=1}^k \eta_{k,j}\frac{x^j}{(1-x)^{j+1}}
\equiv
\frac{\sum_{l=0}^k \gamma_{k,l}x^l}{(1-x)^{k+1}}.
\label{eq.gammadef}
\end{equation}
The first column and the diagonal of Table \ref{bern.tab} appear to be
sequences A097861 and A011818 of the OEIS, respectively \cite{EIS}.
\end{rem}
\begin{table}
\caption{Binomial coefficients $\eta_{k,j}$ of (\ref{bin.eq}).}
\begin{tabular}{r|rrrrrrrr}
$k\backslash j$ & 1 & 2 & 3 & 4 & 5 & 6 & 7 &8\\
\hline
1 & 1 \\
2 & 3 & 3\\
3 & 9 & 24 & 16 \\
4 & 25& 140 & 230 & 115 \\
5 & 70 & 735 & 2250 & 2640 & 1056 \\
6 & 196 & 3675 & 18732 & 38801 & 35322 & 11774 \\
7 & 553 & 17976 & 143696& 468160& 728448& 541184& 154624 \\
8 & 1569& 87024& 1052352& 5067288& 11994354& 14906484& 9350028& 2337507
\end{tabular}
\label{bern.tab}
\end{table}
\begin{rem}
From \eqref{eq.Belb} we deduce the numerator coefficients
defined in \eqref{eq.gammadef}:
\begin{equation}
\gamma_{k,l}=
\sum_{n=0}^l
\binom{k+1}{l-n}(-)^{l-n}
\binom{k+1}{n(k+1)}_{2n}.
\end{equation}
Some of these are shown in Table \ref{gamm.tab}. Caused by the
mirror symmetry of the coefficients,
$-1$ is a root of the polynomial $\sum_l \gamma_{k,l}x^l$ if $k$ is odd;
a factor $1+x$ may then be split off.
\end{rem}
\begin{table}
\caption{Synopsis of the numerators $\gamma_{k,l}$ of the generating functions \eqref{eq.gammadef}.
}
\begin{tabular}{r|rrrrrrrrrrr}
$k\backslash l$ & 0 & 1 & 2 &3 & 4 & 5 & 6 & 7 \\
\hline
1&1&1 \\
2&1&4&1 \\
3&1&15&15&1 \\
4&1&46&136&46&1 \\
5&1&135&920&920&135&1 \\
6&1&386&5405&11964&5405&386&1 \\
7&1&1099&29337&124187&124187&29337&1099&1 \\
8&1&3130&152110&1126258&2112016&1126258&152110&3130\\ 
9&1&8943&767460&9371472&29836764&29836764&9371472&767460\\ 
10&1&25642&3809367&73628622&372715542&626734120&372715542&73628622\\ 
\end{tabular}
\label{gamm.tab}
\end{table}

\begin{table}[htb]
\caption{$A_k^s(n)$
derived from Table \ref{Abconv.tab}, building differences between adjacent columns \cite[A175197]{EIS}. 
}
\begin{tabular}{r|rrrrrrrrr}
$k\backslash n$ & 0 & 1 & 2 &3 &4 &5 &6 &7 & 8\\
\hline
1 &      1&      2&      2&      2&      2&      2&      2&      2&      2\\
2&      1&      6&     12&     18&     24&     30&     36&     42&     48\\
3&      1&     18&     66&    146&    258&    402&    578&    786&   1026\\
4&      1&     50&    330&   1070&   2500&   4850&   8350&  13230&  19720\\
5&      1&    140&   1610&   7580&  23330&  56252& 115850& 213740& 363650\\
6&      1&    392&   7742&  52556& 212436& 635628&1564570&3359440&6521704\\
7&      1&   1106&  37058& 360402&1907458&7071442&20784834&51910994&114945026\\
8&      1&   3138& 177186&2455938&16973970&77854566&273022686&792717990&2001382932
\end{tabular}
\label{Asconv.tab}
\end{table}
Formula (\ref{eq.ADEs}) converts Table \ref{Abconv.tab} into Table \ref{Asconv.tab}.
And similar to Conjecture \ref{conj.Akb} we formulate recurrences
along columns of this derived table:
\begin{conj} (Recurrences of $A_k^s$)
\begin{equation}
(k+1)(k-1) A_k^s(1)-(3k^2-k-1)A_{k-1}^s(1)-k(k-2)A_{k-2}^s(1) +3k(k-1)A_{k-3}^s(1)=0,
\end{equation}
\begin{multline}
2(k-1)(2k+1)(k+1)(65576k-74745)A_k^s(2)
\\
+
(262304k^4-10212201k^3+21353744k^2-8959001k-149490) A_{k-1}^s(2)
\\
+2
(-6440305k^4+44418225k^3-87651471k^2+52631106k-4105233) A_{k-2}^s(2)
\\
+20
(811225k^4-3988621k^3+5814523k^2+2441684k-8566578) A_{k-3}^s(2)
\\
+2
(24847058k^4-190384802k^3+480247197k^2-462996527k+158679414) A_{k-4}^s(2)
\\
-(k-3)
(20387704k^3-72824267k^2-29485137k+331041750) A_{k-5}^s(2)
\\
-10(k-3)(k-4)
(3707581k^2-5729012k+3352341) A_{k-6}^s(2)
\\
+150(k-3)(k-4)(k-5)
(26006k+104375) A_{k-7}^s(2)
=0.
\end{multline}
\end{conj}

\section{Lattice $E_6$} \label{sec.E6}
The task is to sum over the 6-dimensional representation with 
limits set by the 8-dimensional cube:
\begin{equation}
\begin{pmatrix}
0 & 0 & 0& 0 &0 & 1/2  \\
-1 & 0 & 0& 0 &0 & 1/2  \\
1 & -1 & 0& 0 &0 & 1/2  \\
0 & 1 & -1 &  0 &0 & 1/2  \\
0 & 0 & 1 & -1 &  0 & -1/2  \\
0 & 0 & 0 & 1 & -1 & -1/2  \\
0 & 0 & 0 & 0 & 1 & -1/2  \\
0 & 0 & 0 & 0 & 0 & -1/2  \\
\end{pmatrix}
\cdot
\begin{pmatrix}
\alpha_1 \\
\alpha_2 \\
\alpha_3 \\
\alpha_4 \\
\alpha_5 \\
\alpha_6 \\
\end{pmatrix}
=
\begin{pmatrix}
p_1 \\
p_2 \\
p_3 \\
p_4 \\
p_5 \\
p_6 \\
p_7 \\
p_8 \\
\end{pmatrix}
.
\end{equation}
This is extended to an 8-dimensional representation
\begin{equation}
\begin{pmatrix}
0 & 0 & 0& 0 &0 & 1/2 & 0 & 0  \\
-1 & 0 & 0& 0 &0 & 1/2  & 0 & 0 \\
1 & -1 & 0& 0 &0 & 1/2  & 0 & 0 \\
0 & 1 & -1 &  0 &0 & 1/2  & 0 & 0 \\
0 & 0 & 1 & -1 &  0 & -1/2  & 0 & 0 \\
0 & 0 & 0 & 1 & -1 & -1/2  & 0 & 0 \\
0 & 0 & 0 & 0 & 1 & -1/2  & 1 & 0 \\
0 & 0 & 0 & 0 & 0 & -1/2  & 0 & 1 \\
\end{pmatrix}
\cdot
\begin{pmatrix}
\alpha_1 \\
\alpha_2 \\
\alpha_3 \\
\alpha_4 \\
\alpha_5 \\
\alpha_6 \\
\alpha_7 \\
\alpha_8 \\
\end{pmatrix}
=
\begin{pmatrix}
p_1 \\
p_2 \\
p_3 \\
p_4 \\
p_5 \\
p_6 \\
p_7 \\
p_8 \\
\end{pmatrix}
.
\end{equation}
maintaining the count of $E_6^b$ by adding the condition $\alpha_7=\alpha_8=0$ to
the lattice sum. Inversion of this matrix equation yields
\begin{equation}
\begin{pmatrix}
1 & -1 & 0& 0 &0 & 0 & 0 & 0  \\
2 & -1 & -1& 0 &0 & 0  & 0 & 0 \\
3 & -1 & -1& -1 &0 & 0  & 0 & 0 \\
2 & -1 & -1 & -1 &-1 & 0  & 0 & 0 \\
1 & -1 & -1 & -1 &  -1 & -1  & 0 & 0 \\
2 & 0 & 0 & 0 & 0 & 0  & 0 & 0 \\
0 & 1 & 1 & 1 & 1 & 1  & 1 & 0 \\
1 & 0 & 0 & 0 & 0 & 0  & 0 & 1 \\
\end{pmatrix}
\cdot
\begin{pmatrix}
p_1 \\
p_2 \\
p_3 \\
p_4 \\
p_5 \\
p_6 \\
p_7 \\
p_8 \\
\end{pmatrix}
=
\begin{pmatrix}
\alpha_1 \\
\alpha_2 \\
\alpha_3 \\
\alpha_4 \\
\alpha_5 \\
\alpha_6 \\
\alpha_7 \\
\alpha_8 \\
\end{pmatrix}
.
\end{equation}
The first but last equation of this linear system argues that
6 components of $p_i$ are confined to $\sum_{i=2,\ldots,6} p_i=0$
while summing over $|p_i|\le n$ to ensure $\alpha_7=0$; the same sum
regulated the 6-dimensional cube $A_5^b$.
The last equation represents the confinement $p_1+p_8=0$ to
ensure $\alpha_8=0$. Since this
is not entangled with the requirement on the other 6 components, the
associated double sum emits a factor $2n+1$. (Imagine counting points
in a square of edge size $2n+1$ along two coordinates $p_1$ and $p_8$, where
$p_1+p_8=0$ admits only points on the diagonal.)
\begin{thm}(Point counts of $E_6$)
\begin{gather}
E_6^b(n)
= (2n+1)A_5^b(n)
=
\frac{1}{5}(1+2n)^2(5+27n+71n^2+88n^3+44n^4) 
\\
= 1,423,8755,65317,293949,978043,2661919,6277545, 13296601,\ldots;
\nonumber
\\
E_6^s(n)
=
\begin{cases}
1, & n=0;\\
\frac{2}{5}n(47+480n^2+528n^4), & n>0;\\
\end{cases}
\\
= 1,422, 8332, 56562, 228632, 684094, 1683876, 3615626, 7019056, \ldots;
\nonumber \\
E_6^b(x) = \frac{1+416x+5815x^2+12880x^3+5815x^4+416x^5+x^6}{(1-x)^7} .
\end{gather}
\end{thm}

\section{Lattice $E_7$}
The $E_7$ lattice is spanned by
\begin{equation}
\begin{pmatrix}
-1 & 0 & 0 & 0& 0 &0 & 1/2  \\
1& -1 & 0 & 0& 0 &0 & 1/2  \\
0 & 1 & -1 & 0& 0 &0 & 1/2  \\
0 & 0 & 1 & -1 &  0 &0 & 1/2  \\
0 & 0 & 0 & 1 & -1 &  0 & -1/2  \\
0 & 0 & 0 & 0 & 1 & -1 & -1/2  \\
0 & 0 & 0 & 0 & 0 & 1 & -1/2  \\
0 & 0 & 0 & 0 & 0 & 0 & -1/2  \\
\end{pmatrix}
\cdot
\begin{pmatrix}
\alpha_1 \\
\alpha_2 \\
\alpha_3 \\
\alpha_4 \\
\alpha_5 \\
\alpha_6 \\
\alpha_7 \\
\end{pmatrix}
=
\begin{pmatrix}
p_1 \\
p_2 \\
p_3 \\
p_4 \\
p_5 \\
p_6 \\
p_7 \\
p_8 \\
\end{pmatrix}
.
\end{equation}
Again we consider only the sublattice with even $\alpha_7$, that is, integer
$p_i$.
\begin{thm}(Point counts of $E_7$)
\begin{equation}
E_7^b(n) = A_7^b(n).
\end{equation}
\end{thm}
\begin{proof}
We reach out into a direction of the $p_8$ axis adding a unit vector
with axis section $\alpha_8$: $E_7^b(n)$ counts only points with $\alpha_8=0$.
\begin{equation}
\begin{pmatrix}
-1 & 0 & 0 & 0& 0 &0 & 1/2 & 0 \\
1& -1 & 0 & 0& 0 &0 & 1/2 & 0 \\
0 & 1 & -1 & 0& 0 &0 & 1/2 & 0 \\
0 & 0 & 1 & -1 &  0 &0 & 1/2 & 0 \\
0 & 0 & 0 & 1 & -1 &  0 & -1/2 & 0 \\
0 & 0 & 0 & 0 & 1 & -1 & -1/2 & 0 \\
0 & 0 & 0 & 0 & 0 & 1 & -1/2 & 0 \\
0 & 0 & 0 & 0 & 0 & 0 & -1/2 & 1 \\
\end{pmatrix}
\cdot
\begin{pmatrix}
\alpha_1 \\
\alpha_2 \\
\alpha_3 \\
\alpha_4 \\
\alpha_5 \\
\alpha_6 \\
\alpha_7 \\
\alpha_8 \\
\end{pmatrix}
=
\begin{pmatrix}
p_1 \\
p_2 \\
p_3 \\
p_4 \\
p_5 \\
p_6 \\
p_7 \\
p_8 \\
\end{pmatrix}
.
\end{equation}
The inverse of this equation is
\begin{equation}
\begin{pmatrix}
0 & 1 & 1 & 1& 1 &1 & 1 & 0 \\
1& 1 & 2 & 2& 2 &2 & 2 & 0 \\
2 & 2 & 2 & 3& 3 &3 & 3 & 0 \\
3 & 3 & 3 & 3 &  4 &4 & 4 & 0 \\
2 & 2 & 2 & 2 & 2 &  3 & 3 & 0 \\
1 & 1 & 1 & 1 & 1 & 1 & 2 & 0 \\
2 & 2 & 2 & 2 & 2 & 2 & 2 & 0 \\
1 & 1 & 1 & 1 & 1 & 1 & 1 & 1 \\
\end{pmatrix}
\cdot
\begin{pmatrix}
p_1 \\
p_2 \\
p_3 \\
p_4 \\
p_5 \\
p_6 \\
p_7 \\
p_8 \\
\end{pmatrix}
=
\begin{pmatrix}
\alpha_1 \\
\alpha_2 \\
\alpha_3 \\
\alpha_4 \\
\alpha_5 \\
\alpha_6 \\
\alpha_7 \\
\alpha_8 \\
\end{pmatrix}
,
\label{E7inv.eq}
\end{equation}
and---reading the last line---the restriction on the $\alpha_8$ coordinate
implied by the embedding translates into $\sum_i p_i=0$. In comparison,
we can also embed the $A_7$ lattice into its 8-dimensional host,
\begin{equation}
\begin{pmatrix}
1 & 0 & 0 & 0& 0 &0 & 0 & 0 \\
-1& 1 & 0 & 0& 0 &0 & 0 & 0 \\
0 & -1 & 1 & 0& 0 &0 & 0 & 0 \\
0 & 0 & -1 & 1 &  0 &0 & 0 & 0 \\
0 & 0 & 0 & -1 & 1 &  0 & 0 & 0 \\
0 & 0 & 0 & 0 & -1 & 1 & 0 & 0 \\
0 & 0 & 0 & 0 & 0 & -1 & 1 & 0 \\
0 & 0 & 0 & 0 & 0 & 0 & -1 & 1 \\
\end{pmatrix}
\cdot
\begin{pmatrix}
\alpha_1 \\
\alpha_2 \\
\alpha_3 \\
\alpha_4 \\
\alpha_5 \\
\alpha_6 \\
\alpha_7 \\
\alpha_8 \\
\end{pmatrix}
=
\begin{pmatrix}
p_1 \\
p_2 \\
p_3 \\
p_4 \\
p_5 \\
p_6 \\
p_7 \\
p_8 \\
\end{pmatrix}
,
\end{equation}
and invert this representation, too:
\begin{equation}
\begin{pmatrix}
1 & 0 & 0 & 0& 0 &0 & 0 & 0 \\
1& 1 & 0 & 0& 0 &0 & 0 & 0 \\
1 & 1 & 1 & 0& 0 &0 & 0 & 0 \\
1 & 1 & 1 & 1 &  0 &0 & 0 & 0 \\
1 & 1 & 1 & 1 & 1 &  0 & 0 & 0 \\
1 & 1 & 1 & 1 & 1 & 1 & 0 & 0 \\
1 & 1 & 1 & 1 & 1 & 1 & 1 & 0 \\
1 & 1 & 1 & 1 & 1 & 1 & 1 & 1 \\
\end{pmatrix}
\cdot
\begin{pmatrix}
p_1 \\
p_2 \\
p_3 \\
p_4 \\
p_5 \\
p_6 \\
p_7 \\
p_8 \\
\end{pmatrix}
=
\begin{pmatrix}
\alpha_1 \\
\alpha_2 \\
\alpha_3 \\
\alpha_4 \\
\alpha_5 \\
\alpha_6 \\
\alpha_7 \\
\alpha_8 \\
\end{pmatrix}
.
\end{equation}
The implied slice $\alpha_8=0$ and the last line of this equation leads to
the same condition $\sum_i p_i=0$ as derived from (\ref{E7inv.eq}).
Since both cases
select from the $(2n+1)^8$ points in the hypercube subject to the same
condition, both counts are the same.
\end{proof}

\section{Lattice $E_8$} \label{sec.E8}
The $E_8$ coordinates are mediated by
\begin{equation}
\begin{pmatrix}
2& -1 & 0 & 0 & 0& 0 &0 & 1/2  \\
0 & 1& -1 & 0 & 0& 0 &0 & 1/2  \\
0 & 0 & 1 & -1 & 0& 0 &0 & 1/2  \\
0 & 0 & 0 & 1 & -1 &  0 &0 & 1/2  \\
0 & 0 & 0 & 0 & 1 & -1 &  0 & -1/2  \\
0 & 0 & 0 & 0 & 0 & 1 & -1 & -1/2  \\
0 & 0 & 0 & 0 & 0 & 0 & 1 & -1/2  \\
0 & 0 & 0 & 0 & 0 & 0 & 0 & -1/2  \\
\end{pmatrix}
\cdot
\begin{pmatrix}
\alpha_1 \\
\alpha_2 \\
\alpha_3 \\
\alpha_4 \\
\alpha_5 \\
\alpha_6 \\
\alpha_7 \\
\alpha_8 \\
\end{pmatrix}
=
\begin{pmatrix}
p_1 \\
p_2 \\
p_3 \\
p_4 \\
p_5 \\
p_6 \\
p_7 \\
p_8 \\
\end{pmatrix}
.
\label{E8.eq}
\end{equation}
Explicit numbers are found with the formula in Theorem \ref{V.thm}:
\begin{thm}(Lattice points in the bulk and on the surface of $E_8$)
\begin{gather}
E_8^b(n) = V_8^g(n)
= 1,3281,195313,2882401,21523361\ldots \\
E_8^s(n) = 
\begin{cases}
1, & n=0 ;\\
16n(4n^2+1)(16n^4+24n^2+1), & n>0 ;\\
\end{cases}
\\
=1,3280,192032,2687088,18640960,85656080,\ldots
\nonumber
\end{gather}
\end{thm}
\begin{proof}
The inverse of the generator matrix in \eqref{E8.eq} has exactly one row filled
with the value $1/2$, all other entries are integer. As already argued for
the $D$-lattices in sections \ref{D3.sec}--\ref{D4.sec}, this leads to the constraint that the
sum over the $p_i$ must remain even, which matches Definition \ref{V.def}.
\end{proof}

\section{Summary} 
For $D_k$ lattices, the number of lattice points inside a hypercube
is essentially a $k$-th order polynomial of the edge length, summarized in Eq.\ \eqref{eq.Dkbn}.
For $A_k$ lattices, explicit polynomials have been computed for $k\le 5$
in Eqs.\ \eqref{eq.A2}, \eqref{eq.A3}, \eqref{eq.A4} and \eqref{eq.A5}.
For higher dimensions, the numbers are centered multinomial coefficients
\eqref{eq.Belb} which can be quickly converted to
$k$-th order polynomials in $n$.
The counts for $E_6$, $E_7$ and $E_8$ are closely associated
with the counts for $A_5$, $A_7$ and $D_8$, respectively.

\bibliographystyle{amsplain}
\bibliography{all}

\end{document}